\documentclass[a4paper,12pt]{article}

\usepackage{amsmath}
\usepackage{amsfonts}
\usepackage{amssymb}
\usepackage{graphicx}
\usepackage{amsthm}
\usepackage{varioref}

\labelformat{prop}{Proposition~#1}
\newtheorem{theo}{Theorem}[section]
\labelformat{theo}{Theorem~#1}
\newtheorem{cor}{Corollary}[section]
\labelformat{cor}{Corollary~#1}

\newtheorem{claim}{Claim}[section]
\labelformat{claim}{Claim~#1}

\labelformat{lemma}{Lemma~#1}
\labelformat{fact}{Fact~#1}

\def\S{{\sf S}}
\def\s{{\sf s}}
\def\M{\mathcal{M}}
\def\P{\mathcal{P}}

\title{Basic Packing of Arborescences}

\author{Olivier Durand de Gevigney\thanks {Laboratoire G-SCOP, CNRS, Grenoble INP, UJF,  46, Avenue F\'elix Viallet, Grenoble, France, 38000.} \\ Viet-Hang Nguyen$^*$ \\ Zolt\'an Szigeti$^*$ \footnote{partially supported by the TEOMATRO grant ANR-10-BLAN 0207}}

\begin{document}
\maketitle

\abstract{We provide the directed counterpart of a slight extension of Katoh and Tanigawa's result \cite{Kato_Tanigawa} on rooted-tree decompositions with matroid constraints. Our result characterises digraphs having a packing of arborescences with matroid constraints. It is a proper extension of Edmonds' result \cite{edmonds} on packing of spanning arborescences and implies -- using a general orientation result of Frank \cite{frank} -- the above result of Katoh and Tanigawa.

We also give a complete description of the convex hull of the incidence vectors of the basic packings of arborescences and prove that the mimimum cost version of the problem can be solved in polynomial time.}


\section{Introduction}

Let $G=(V,E)$ be a graph. For a vertex set $X$, $E(X)$ denotes the set of edges of $G$ with both extremities  in $X$.
 We say that $G$ is a \emph{rooted-tree} or more precisely a \emph{tree rooted at $r$} if $G$ is connected and cycle free and $r$ is a vertex of $G.$ We note that a tree rooted at $r$ may consist of only the vertex $r$ and no edges. Note also that  a tree can be  rooted at any vertex of its. 
\medskip

Our starting point is the  result of Tutte \cite{Tutte} and Nash-Williams  \cite{NW} on packing of spanning trees.
For a partition ${\cal P}$ of $V,$ $e_G({\cal P})$ denotes the number of edges of $G$ between the different members of ${\cal P}.$ We always suppose  that the  members of ${\cal P}$ are not empty. Following Frank \cite{frankbook}, $G$ is called \emph{$k$-partition-connected} if 
\begin{eqnarray}
e_G(\P) \geq k(|\P|-1) &&\text{ for every partition  $\P$ of $V.$}\label{tuttecond}
\end{eqnarray}

\begin{theo} [Tutte \cite{Tutte}, Nash-Williams \cite{NW}]\label{tuttetheorem}
A graph $G=(V,E)$ contains $k$ edge-disjoint spanning trees (rooted at a vertex  $r$  of $G$) if and only if  
$G$ is $k$-partition-connected.
\end{theo}

Let $D=(V,A)$ be a digraph. 
For a vertex set $X$, $D[X]$ denotes the induced subgraph of $D$ on $X.$
 We say that a vertex $v$ is \emph{reachable} from a vertex $u$ in $D$ if there exists a directed path from $u$ to $v$ in $D.$ For convenience, we will not distinguish the vertex $v$ from the set $\{v\}.$  
For a vertex set $X$, we denote by $\varrho_D(X)$ the set of arcs entering $X$ and we define $\rho _D(X)=|\varrho_D(X)|$. We say that $D$ is an \emph{arborescence rooted at $r$} if $D$ is a directed tree, $r$ is a vertex of $D$ of in-degree $0$ and all the other vertices of $D$ are of in-degree $1.$ We note that an arborescence rooted at $r$ may consist of only the vertex $r$ and no arcs. Note also that an arborescence has a unique root. It is well-known that $D$ contains a spanning arborescence rooted at a vertex  $r$  of $D$ if and only if  every non-empty vertex set not containing $r$ has in-degree at least $1.$
\medskip

The directed counterpart of \ref{tuttetheorem} is the result of Edmonds \cite{edmonds}  on packing of spanning arborescences. 

\begin{theo}[Edmonds \cite{edmonds}]\label{edmondstheorem}
A digraph $D=(V,A)$ contains $k$ arc-disjoint spanning arborescences rooted  at a vertex  $r$  of $D$ if and only if 
\begin{eqnarray}
\rho _D(X) \geq k &&\text{ for all non-empty  $X\subseteq V\setminus r.$}\label{edmondscond}
\end{eqnarray}
\end{theo}

Frank \cite{frank1} showed how to deduce \ref{tuttetheorem} from \ref{edmondstheorem}. He proved that (\ref{tuttecond}) is the necessary and sufficient  condition for the undirected graph $G$ to have an orientation $D$ that satisfies (\ref{edmondscond}). Then, by \ref{edmondstheorem}, $D$ contains $k$ arc-disjoint spanning arborescences rooted at  $r$ that provide the $k$ edge-disjoint spanning trees rooted at  $r$ in $G.$
\medskip

Let $\S=\{\s_1,\ldots, \s_t\}$ be a set and $\pi$ a map from $\S$ to $V$. We may think of $\pi$ as a placement of the elements of $\S$ at vertices of $V$ and different elements of $\S$ may be placed at the same vertex. In this paper $t$ will always denote the size of $\S.$ 
The triplet $(G,\S,\pi)$ (respectively $(D,\S,\pi)$) is called a \emph{graph} (resp. \emph{digraph}) \emph{with roots}.
For $X \subseteq V$, we denote by $\S_X$ the set $\pi^{-1}(X)$.
\medskip

A function  $p:2^\Omega \rightarrow \mathbb Z$  is called \emph{supermodular} (respectively \emph{intersecting supermodular}) if for all $X,Y\subseteq \Omega$ (resp. for all $X,Y\subseteq \Omega$ that are {intersecting}),
\begin{eqnarray*}
p(X)+p(Y) & \leq  & p(X\cap Y)+p(X\cup Y).\label{alf2}
\end{eqnarray*}
 A function $b:2^\Omega \rightarrow \mathbb Z$ is called \emph{submodular} if $-b$ is supermodular. Note that the in-degree function $\rho_D$ of a digraph $D$  is submodular.
\medskip

Let $\M$ be a matroid on $\S$ with  rank function  $r_{\M}$. It is well-known that $r_{\M}$ is monotone non-decreasing and submodular. A set ${\sf Q}\subseteq {\S}$ is \emph{independent}  if $r_{\M}({\sf Q})=|{\sf Q}|$. Recall that  every subset of an independent set is independent. A maximal independent set is a \emph{base} of $\M.$ Each base has the same size, namely $r_{\M}({\S}).$ $\M$ is called a \emph{free matroid} if each subset of $\S$ is independent.
For a set ${\sf Q}\subseteq {\S},$ we define ${\rm Span}_{\M}({\sf Q})=\{{\s}\in {\S}:r_{\M}({\sf Q}\cup \{{\s}\})=r_{\M}({\sf Q})\}.$ The set ${\sf Q}$ is called a \emph{spanning set} of $\M$ if ${\rm Span}_{\M}({\sf Q})={\S}.$
\medskip

The following definition was introduced by Katoh and Tanigawa \cite{Kato_Tanigawa}. An \emph{$\M$-basic packing of rooted-trees} is a set $\{T_1,\dots ,T_t\}$ of  pairwise edge-disjoint trees  such that for  $i=1,\dots ,t$, $T_i$ is rooted at $\pi({\s}_i)$ and, for each $v \in V$, the set $\{{\s}_i \in {\S} : v \in V(T_i)\}$ forms a base of $\M$. For the sake of convenience, we say that $T_i$ is rooted at ${\s}_i$. Note that the trees are not necessarily spanning and each vertex of $G$ belongs to  exactly $r_{\M}(\S)$ trees.
\medskip

The following result characterizes graphs with roots that have  a basic packing of rooted-trees. 
It will be derived from its directed counterpart (\ref{packing}) at the end of this section.
We say that the map $\pi$ is \emph{$\M$-independent} if ${\S}_v$ is independent in $\M$ for all $v \in V.$
The graph with roots $(G,{\S},\pi)$ is called \emph{$\M$-partition-connected} if 
\begin{eqnarray*}
 e_G(\P) \geq r_\M({\S})|\P| - \sum_{X \in \P}r_\M({\S}_X) && \text{ for every partition  $\P$ of $V.$}\label{ptcond2}
\end{eqnarray*}
\begin{theo} \label{thm:packing_non_oriented}
 Let $(G,{\S},\pi)$ be a graph with roots and  $\M$  a matroid on ${\S}$. 
 There exists an $\M$-basic packing of rooted-trees in $(G,{\S},\pi)$ if and only if $\pi$ is $\M$-independent and $(G,{\S},\pi)$ is $\M$-partition-connected.
\end{theo}


If $\M$ is the free matroid then the problem of $\M$-basic packing of rooted-trees and that of packing of  spanning trees coincide. Hence \ref{thm:packing_non_oriented} is a proper extension of \ref{tuttetheorem}.
It is not difficult to see that this theorem easily implies the following theorem of Katoh and Tanigawa \cite{Kato_Tanigawa}.
 A \emph{rooted-component} of $(G,{\S},\pi)$ is a pair $(C,{\s})$ where  $C$ is a connected subgraph of $G$ and ${\s}\in {\S}_{V(C)}$. 
 
\begin{theo}[Katoh and Tanigawa \cite{Kato_Tanigawa}]
\label{thm:KT_5.1}
 Let $G=(V,E)$ be a graph, ${\S}=\{{\s}_1,\ldots,{\s}_t\}$ a set, $\pi$ a placement of $\S$ in $V$ and $\M$   a matroid on ${\S}$. Then $(G,{\S},\pi)$ admits rooted-components $(C_1,{\s}_1),$ $\ldots,$ $(C_t,{\s}_t)$ such that $E = \cup_{i=1}^t E(C_i)$ and the set $\{{\s}_i\in {\S} : v\in V(C_i)\}$ is a spanning set of $\M$ for every $v\in V$ if and only if $(G,{\S},\pi)$ is $\M$-partition-connected.
\end{theo}

Katoh and Tanigawa deduced  \ref{thm:KT_5.1} from the following dual form of its. We show that \ref{thm:packing_non_oriented} also implies \ref{thm:KT_1.2}.

\begin{theo}[Katoh and Tanigawa \cite{Kato_Tanigawa}]
 \label{thm:KT_1.2}
 Let $G=(V,E)$ be a graph, ${\S}=\{\s_1,\ldots,{\s}_t\}$ a set and $\pi$ a placement of $\S$ in $V$. Let $\M$ be a matroid on ${\S}$   of rank $k$ with rank function $r_{\M}$. Then $(G,{\S},\pi)$ admits an $\M$-basic packing of rooted-trees such that $E$ is the union of the edge sets of these trees if and only if  $\pi$ is $\M$-independent, $|E|+|{\S}|=k|V|$ and 
$|F|+|{\S}_{V(F)}|\leq k|V(F)|-k+r_\M({\S}_{V(F)})$ for all non-empty $F\subseteq E.$
\end{theo}
 
\begin{proof}
The necessity of the conditions is pretty straightforward as one can see in \cite{Kato_Tanigawa}. 

Now suppose that the conditions hold. 
%
For every partition $\P$ of $V$, by the inequality applied for $E(X) \ (X\in {\P})$ and by the equality,
 $e_G(\P)=|E|-\sum_{X\in \P}|E(X)|
 \geq |E|- \sum_{X\in \P}(k|X|-k+r_\M({\S}_X)-|{\S}_X|)
 = k|\P|-\sum_{X\in\P}r_\M({\S}_X). $
Hence, $\pi$ is $\M$-independent and $(G,{\S},\pi)$ is $\M$-partition-connected. Then \ref{thm:packing_non_oriented} implies that $(G,{\S},\pi)$ admits an $\M$-basic packing of rooted-trees and, by $|E|+|{\S}|=k|V|$,  $E$ is the union of the edge sets of the trees in the packing.
\end{proof}

The main contribution of the present paper is to mimic Frank's approach (mentioned above on packing of spanning trees) for basic packing of rooted-trees. We provide the directed counterpart of \ref{thm:packing_non_oriented}, a short proof of it and we show that it implies \ref{thm:packing_non_oriented}  (and hence \ref{thm:KT_5.1} and \ref{thm:KT_1.2}) via an orientation theorem of Frank.
\medskip

Inspired by the definition of Katoh and Tanigawa, we define an \emph{$\M$-basic packing of arborescences} as a set  $\{T_1,\dots ,T_t\}$ of pairwise arc-disjoint arborescences  such that  for  $i=1,\dots ,t$,  $T_i$ is rooted at $\pi({\s}_i)$ and, for each $v \in V$, the set $\{{\s}_i \in {\S} : v \in V(T_i)\}$ forms a base of $\M$. We also say that $T_i$ is rooted at ${\s}_i$. For a better understanding, let us mention that the arborescences are not necessarily spanning and each vertex of $D$ belongs to exactly $r_{\M}(\S)$ arborescences. 

\medskip

Our main result is the following theorem. The digraph with roots $(D,{\S},\pi)$ is called \emph{$\M$-connected} if 
\begin{eqnarray}
\rho_D(X) \geq r_{\M}({\S}) - r_{\M}({\S}_X) &&\text{ for all non-empty } X \subseteq V.\label{cond2}
\end{eqnarray}

\begin{theo} \label{packing} Let $(D,{\S},\pi)$ be a digraph with roots and  $\M$  a matroid on ${\S}$. There exists an $\M$-basic packing of arborescences in $(D,{\S},\pi)$ if and only if  $\pi$ is $\M$-independent and $(D,{\S},\pi)$ is $\M$-connected.
\end{theo}


If $\M$ is the free matroid and $\pi$ places every element of $\S$ at a single vertex $r$ of $D$ then the problem of $\M$-basic packing of arborescences and that of  packing of  spanning arborescences rooted at $r$ coincide. Hence \ref{packing} is a proper extension of \ref{edmondstheorem}.

\medskip

Let us recall the following general orientation result of Frank \cite{frank}.

\begin{theo}[Frank \cite{frank}]\label{frank_theo}
 Let $G=(V,E)$ be a graph and  $h : 2^V \rightarrow \mathbb{Z}_+$  an intersecting supermodular non-negative non-increasing set-function such that $h(\emptyset)=h(V)=0$. There exists an orientation $D$ of $G$ such that $\rho_D(X) \geq h(X)$ for all non-empty $X \subset V$ if and only if for every partition  $\P$ of $V$,

 \[ e_G(\P) \geq \sum_{X \in \P}h(X).\] 
\end{theo}
%
%
\ref{frank_theo} immediately implies  the following corollary by taking $h(X) = r_{\M}({\S}) - r_{\M}({\S}_X)$ if $X$ is not empty and $h(\emptyset)=0$.
\begin{cor}\label{orientation}
 Let $(G,{\S},\pi)$ be a graph with roots and  $\M$  a matroid on ${\S}$. There exists an orientation $D$ of $G$ such that $(D,{\S},\pi)$ is $\M$-connected if and only if $(G,{\S},\pi)$ is $\M$-partition-connected.
 \end{cor}

Let us show that  \ref{orientation} and \ref{packing}    imply \ref{thm:packing_non_oriented}.

\begin{proof} {\it (of \ref{thm:packing_non_oriented})}
First suppose that  there exists an $\M$-basic packing $\{T_1,\dots ,T_t\}$ of rooted-trees in $(G,{\S},\pi)$. Let $D$ be  an orientation of $G$  where each tree $T_i$ rooted in ${\s}_i$ becomes an arborescence $T'_i$ rooted in ${\s}_i$. Then   $\{T'_1,\dots ,T'_t\}$ is  an $\M$-basic packing of arborescences in $(D,{\S},\pi)$. By \ref{packing},  $\pi$ is $\M$-independent and $(D,{\S},\pi)$ is $\M$-connected. Hence, by \ref{orientation}, $G$ is $\M$-partition-connected.

Now suppose that $\pi$ is $\M$-independent  and $(G,{\S},\pi)$ is $\M$-partition-connected. By 
\ref{orientation}, there exists an orientation $D$ of $G$ such that $(D,{\S},\pi)$ is $\M$-connected. Then, by \ref{packing}, there exists  an $\M$-basic packing of arborescences in $(D,{\S},\pi)$ which provides, by forgetting the orientation, an $\M$-basic packing  of rooted-trees in $(G,{\S},\pi)$.
\end{proof}

\section{Proof of the main theorem}

First we prove the necessity of the conditions.
\begin{proof} {\it (of necessity in \ref{packing})}
Suppose that there exists an $\M$-basic packing $\{T_1,\dots ,T_t\}$ of arborescences in $(D,{\S},\pi)$. Let $v$ be an arbitrary vertex of $V$ and $X$ a vertex set containing $v.$ Then ${\sf B}:=\{{\s}_i \in {\S} : v \in V(T_i)\}$ forms a base of $\M.$ Let ${\sf B}_1={\sf B}\cap {\S}_X$ and ${\sf B}_2={\sf B}\setminus {\S}_X.$ Then, since ${\sf B}_1$ is independent in $\M$ and ${\S}_v\subseteq {\sf B}_1,$  $\pi$ is $\M$-independent. Moreover, since $r_{\M}$ is monotone, $|{\sf B}_1|=r_{\M}({\sf B}_1)\leq r_{\M}({\S}_X)$. For each root ${\s}_i\in {\sf B}_2,$ there exists an arc of $T_i$ that enters $X$ and the arborescences are arc-disjoint, so we have $\rho _D(X)\geq |{\sf B}_2|=|{\sf B}|-|{\sf B}_1|\geq r_{\M}(\S)-r_{\M}({\S}_X)$ that is $(D,{\S},\pi)$ is $\M$-connected.
\end{proof}

Before proving the sufficiency of the conditions we establish two technical claims.

\begin{claim}\label{span_lemma}
 Let $\M$ be a matroid on ${\S}$ with rank function $r_{\M}$ and ${\sf P},{\sf Q} \subseteq {\S}$ such that $r_{\M}({\sf P} \cap {\sf Q}) + r_{\M}({\sf P} \cup {\sf Q}) = r_{\M}({\sf P}) + r_{\M}({\sf Q})$ and ${\s} \in {\rm Span}_{\M}({\sf P}) \cap {\rm Span}_{\M}({\sf Q}).$ Then ${\s}\in {\rm Span}_{\M}({\sf P} \cap {\sf Q}).$
\end{claim}

\begin{proof}
By the monotonicity and submodularity of the rank function and by the assumptions, 
$r_{\M}({\sf P} \cap {\sf Q}) + r_{\M}({\sf P} \cup {\sf Q}) \leq  r_{\M}(({\sf P}\cap {\sf Q})\cup {\s}) + r_{\M}(({\sf P} \cup {\sf Q})\cup {\s})
  \leq  r_{\M}({\sf P}\cup {\s}) + r_{\M}({\sf Q}\cup {\s})
 =  r_{\M}({\sf P}) + r_{\M}({\sf Q})
 =  r_{\M}({\sf P} \cap {\sf Q}) + r_{\M}({\sf P} \cup {\sf Q}). $
Hence equality holds everywhere, in particular $r_{\M}({\sf P} \cap {\sf Q})=r_{\M}(({\sf P}\cap {\sf Q})\cup {\s})$, that is ${\s} \in {\rm Span}_{\M}({\sf P} \cap {\sf Q}).$
\end{proof}

Let us introduce the following definitions.
A vertex set $X$ is called \emph{tight} if 
$\rho_D(X) = r_{\M}(\S) - r_{\M}({\S}_X)$. For vertex sets $X$ and $Y$, we say that $Y$ \emph{dominates} $X$ if ${\S}_X\subseteq {\rm Span}_{\M}({\S}_Y).$ Note that since, for ${\sf Q}\subseteq {\S},$ ${\rm Span}_{\M}({\rm Span}_{\M}({\sf Q}))={\rm Span}_{\M}({\sf Q})$, domination is a transitive relation.
We say that an arc $uv$ is \emph{good} if $v$ dominates $u$, otherwise it is \emph{bad}.

\begin{claim}\label{tight}
Suppose that $(D,\S,\pi)$ is $\M$-connected. Let $X$  be a tight set and $v$ a vertex of $X.$
\begin{itemize}
\item [(a)] If $Y$ is a tight set that contains $v$, then  $X \cap Y$ and $X \cup Y$ are tight and $r_{\M}({\S}_{X} \cap {\S}_{Y}) + r_{\M}({\S}_{X} \cup {\S}_{Y}) = r_{\M}({\S}_X) + r_{\M}({\S}_Y).$
\item [(b)] If  $Y$ is the set of vertices of $X$ from which  $v$  is reachable in $D[X]$, then $v\in Y\subseteq X$, $Y$ is tight and dominates $X.$
\item [(c)] If  $Y$ is the set of vertices of $X$ from which  $v$  is reachable in $D[X]$ using only good arcs, then  $v$ dominates $Y.$ 
\end{itemize}

\end{claim}
\begin{proof}
(a) By the submodularity of  $r_{\M}$, tightness of $X$ and $Y$, the submodularity of $\rho_D$,  $X \cap Y \neq \emptyset$ and (\ref{cond2}),
 $ r_{\M}({\S}_{X \cap Y}) + r_{\M}({\S}_{X \cup Y})  =  r_{\M}({\S}_{X} \cap {\S}_{Y}) + r_{\M}({\S}_{X} \cup {\S}_{Y}) 
  \leq  r_{\M}({\S}_X) + r_{\M}({\S}_Y) 
   =  r_{\M}(\S) - \rho_D(X) + r_{\M}(\S) - \rho_D(Y) 
   \leq  r_{\M}(\S) - \rho_D(X \cap Y) + r_{\M}(\S) - \rho_D(X \cup Y) 
   \leq  r_{\M}({\S}_{X\cap Y}) + r_{\M}({\S}_{X \cup Y}).
$
 Hence equality holds everywhere and (a) follows.
 \medskip
 
(b) By the definition of $Y$, $v\in Y\subseteq X$ and every arc that enters $Y$ enters $X$ as well. Then, by (\ref{cond2}), the tightness of $X$ and the monotonicity of $ r_{\M}$, we have $r_{\M}(\S) - r_{\M}({\S}_Y)  \leq  \rho_D(Y)
  \leq  \rho_D(X)=  r_{\M}(\S) - r_{\M}({\S}_X) \leq  r_{\M}(\S) - r_{\M}({\S}_Y).$
   Thus equality holds everywhere and (b) follows.
   \medskip
 
(c)  For all $y\in Y$, there exists a directed path  $y=v_l, \dots , v_1=v$ from $y$ to $v$ in $D[X]$ using only good arcs. Then ${\S}_y={\S}_{v_l}\subseteq \dots \subseteq {\rm Span}_{\M}({\S}_{v_1})= {\rm Span}_{\M}({\S}_v)$. Hence ${\S}_Y = \bigcup_{y \in Y} {\S}_y\subseteq {\rm Span}_{\M}({\S}_v)$ and (c) follows.
\end{proof}

Now we can prove the main result.

\begin{proof}  {\it (of sufficiency in \ref{packing})} We start by proving the following claim.

\begin{claim}\label{arcs_are_good}
 If there is no bad arc then taking $|{\S}_v|$ times  each vertex $v$  gives an $\M$-basic packing of arborescences in $(D,{\S},\pi)$.
\end{claim}
\begin{proof}
For every vertex $v,$ let us denote by $Z_v$ the set of vertices  from which  $v$  is reachable in $D.$ 
Since $V$ is tight, \ref{tight}(b) implies that $Z_v$ dominates $V.$ Moreover, since every arc is good, by \ref{tight}(c), $v$ dominates ${Z_v}$ and hence, since $\pi$ is $\M$-independent, ${\S}_v$ is a base of $\M$ for all $v \in V$.
\end{proof}

We now prove the sufficiency by induction on $|A|$. If $A$ is empty, then there is no bad arc, and, by \ref{arcs_are_good}, the theorem is proved.

\medskip
 So we may assume that $A$ is not empty and there exists at least one bad arc.
\medskip

 For a bad arc $uv \in A$ and ${\s} \in {\S}_u \setminus Span(\S_v)$, let $D'=D-uv,$ ${\S}'$  the set obtained by adding a new element ${\s}'$ to ${\S}$,  $\M'$  the matroid on ${\S}'$ obtained from $\M$ by considering ${\s}'$ as an element parallel to ${\s}$ and $\pi'$  the placement of $\S'$ in $V$ obtained from $\pi$ by placing the new element ${\s}'$ at $v$.
 \medskip

By choice of $\s$, $\pi'$ is $\M'$-independent. If the  digraph with roots $(D',{\S}',\pi')$ is $\M'$-connected, then, by induction, there exists an $\M'$-basic packing $\P'$ of arborescences in $(D',{\S}',\pi')$. Since ${\s}$ and ${\s}'$ are parallel in $\M'$,   the arborescences $T$ and $T'$ of $\P'$  rooted at  ${\s}$ and ${\s}'$ are vertex disjoint, so $T'' = T \cup T' \cup uv$ is an arborescence rooted at ${\s}$. Then $(\P' \cup \{T''\} )\setminus \{T,T'\}$ is an $\M$-basic packing of arborescences in $(D,{\S},\pi)$. Hence the proof of the theorem is reduced to the proof of the following claim.
\medskip
\begin{claim} \label{nice_bad_arc}
 There exist a bad arc $uv$ and $\s \in \S_u \setminus Span(\S_v)$ such that $(D',{\S}',\pi')$ is $\M'$-connected.
\end{claim}
\begin{proof} Assume that the claim is false. Let $uv\in A$ be a bad arc and $\s \in \S_u \setminus Span(\S_v)$, by assumption, there exists $\emptyset \neq X_{\s} \subset V$ such that $\rho_{D'}(X_{\s}) < r_{\M}(\S) - r_{\M'}({\S}'_{X_{\s}}).$
Hence, by (\ref{cond2}) and the monotonicity of $r_{\M'}$, $\rho_{D'}(X_{\s})+1\geq \rho_{D'}(X_{\s}) +\rho_{uv}(X_{\s}) = \rho_D(X_{\s}) \geq r_{\M}(\S) - r_{\M}({\S}_{X_{\s}})\geq r_{\M}(\S) - r_{\M'}({\S}'_{X_{\s}})\geq \rho_{D'}(X_{\s})+1,$ so equality holds everywhere and hence  $uv$ enters $X_{\s}$, $X_{\s}$ is tight and ${\s} \in {\rm Span}_{\M}({\S}_{X_{\s}})$. Hence, by \ref{tight}, $X=\cup_{\s \in \S_u \setminus Span(\S_v)}X_\s$ is tight and, by $v \in X$, $\S_u = (\S_u \setminus Span(\S_v)) \cup (\S_u \cap Span(\S_v)) \subseteq Span(\S_X) \cup Span(\S_X) = Span(\S_X)$. So we proved that
 \begin{eqnarray}
\textrm{  every bad arc } uv \textrm{ enters a tight set } X \textrm{  that dominates } u.\label{condviol2'}
 \end{eqnarray}

Among all pairs $(uv,X)$ satisfying (\ref{condviol2'}) choose one with $X$ minimal.  
\medskip

Suppose that every arc in $D[X]$ is good. Note that, by \ref{tight}(b) and the minimality of $X$, $v$ can be reached from all vertices of $X$  in $D[X]$. Then, by (\ref{condviol2'}), $X$ dominates  $u$ and, by \ref{tight}(c), $v$ dominates $X$ so $v$ dominates $u$ which contradicts the fact that $uv$ is bad. 
\medskip

Hence there exists a bad arc $u'v'$ in $D[X]$. Then, by (\ref{condviol2'}), $u'v'$ enters a tight set $Y$  that dominates $u'.$ By $v' \in X\cap Y$, the tightness of $X$ and $Y$, $u'\in X,$ ${\S}_{u'} \subseteq {\rm Span}_{\M}({\S}_Y)$, \ref{tight}(a) and \ref{span_lemma}, we have that $X \cap Y$ is tight and ${\S}_{u'} \subseteq {\rm Span}_{\M}({\S}_X \cap {\S}_Y) = {\rm Span}_{\M}({\S}_{X \cap Y})$. Since the bad arc $u'v'$ enters the tight set $X \cap Y$ that dominates $u'$ and $X\cap Y$ is a proper subset of $X$ (since $u'\in X\setminus Y$), this contradicts the minimality of $X$. 
\end{proof} \end{proof}

\section{Polyhedral aspects}

In this section we study a polyhedron describing the basic packings of arborescences. 
\medskip

We need the following general result of Frank \cite{frank2}.

\begin{theo}[Frank \cite{frank2}]\label{frank_theo2}
 Let $D=(V,A)$ be a digraph, $p:2^V \rightarrow \mathbb{Z}_+$ a non-negative intersecting supermodular set-function such that $\rho_D(Z) \geq p(Z)$ for every $Z \subseteq V$. Then the polyhedron defined by the following  linear system is integer:
\begin{eqnarray*}
1 \geq x(a) \geq 0 && \text{ for all } a \in A, \\
x(\varrho_D(X)) \geq p(X) && \text{ for all non-empty } X \subseteq V.
\end{eqnarray*}
\end{theo}

This following theorem is a corollary of \ref{packing} and \ref{frank_theo2}.

\begin{theo} \label{polyhedra_description} Let $(D=(V,A),{\S},\pi)$ be a digraph with roots and  $\M$  a matroid on ${\S}$  of rank $k$ with rank function $r_{\M}$. There exists an $\M$-basic packing of arborescences in $(D,{\S},\pi)$ if and only if the polyhedron $P_{\M,D}$ defined by the linear system
\begin{eqnarray}
1 \geq x(a) \geq 0 && \text{ for all } a \in A, \label{P_eq1}\\
x(\varrho_D(X)) \geq k - r_{\M}(\S_X) && \text{ for all non-empty } X \subseteq V, \label{P_eq2}\\
x(A) = k|V| - |\S| \label{P_eq3}
\end{eqnarray}
 is not empty. In this case, $P_{\M,D}$ is integer and its vertices are the characteristic vectors of the arc sets of the $\M$-basic packings of arborescences in $(D,\S,\pi)$.
\end{theo}

\begin{proof} Suppose there exists an $\M$-basic packing of arborescences in $(D,\S,\pi)$ and call $A' \subseteq A$ its arc set. Let $x$ be the characteristic vector of $A'$. We have $x(A)=|A'| = \sum_{v \in V} \rho_{A'}(v) = \sum_{v \in V} (k-|\S_v|)  = k|V| - |\S|$ and $x(\varrho_D(X)) = \rho_{A'}(X) \geq k - r_{\M}(\S_X)$ for all non-empty $X \subseteq V$ by (\ref{cond2}). So $x \in P_{\M,D}$.

Now suppose that $P_{\M,D}$ is not empty. 
Since the function $k - r_{\M}(\S_X)$ is non-negative intersecting supermodular and, by (\ref{P_eq1}) and (\ref{P_eq2}),   $\rho_D(X) \geq k - r_{\M}(\S_X)$ for all non-empty $X \subseteq V$, \ref{frank_theo2} implies that the polyhedron $P$ described by (\ref{P_eq1}) and (\ref{P_eq2}) is integer. By (\ref{P_eq2}), for all $x \in P$, 
\begin{equation}
 \label{eqn:face}
 x(A) =\sum_{v \in V} x(\varrho_D(v)) \geq \sum_{v \in V} (k-r_{\M}(\S_v)) \geq \sum_{v \in V} (k-|\S_v|) = k|V| - |\S|,
\end{equation}

 \noindent that is, $x(A)\geq k|V|-|\S|$ is a valid inequality for $P$. Then, by (\ref{P_eq3}), $P_{\M,D}$ is a face of the integer polyhedron $P$ and hence $P_{\M,D}$ is also integer. Furthermore, for $x \in P_{\M,D}$, equality holds everywhere in (\ref{eqn:face}), thus, $|\S_v| = r_{\M}(\S_v)$ for all $v \in V$ and hence $\pi$ is $\M$-independent. A vertex $x$ of $P_{\M,D}$ defines an arc set $A'=\{a \in A, x(a)=1\}$. By (\ref{P_eq2}), the digraph with roots  $((V,A'),\S,\pi)$ is $\M$-connected. Therefore, by \ref{packing}, there exists an $\M$-basic packing of arborescences in $((V,A'),\S,\pi)$ whose arc set is, by (\ref{P_eq3}), equal to $A'$, and the theorem follows.
\end{proof}

\section{Algorithmic aspects}

We use the following theorem proved by Iwata, Fleischer and Fujishige \cite{Iwata} and independently by Schrijver \cite{schrijver}.
\begin{theo}[Iwata, Fleischer and Fujishige \cite{Iwata}, Schrijver \cite{schrijver}]\label{minimizing_sub}
A submodular function can be minimized in polynomial time. 
\end{theo}
 
 In this section we assume that a matroid is given by an oracle for the rank function.
 The following theorem is a corollary of \ref{minimizing_sub} and \ref{packing}.
\begin{theo} \label{packing_polytime}
 Let $(D,{\S},\pi)$ be a digraph with roots and $\M$  a matroid on ${\S}$. An $\M$-basic packing of arborescences in $(D,{\S},\pi)$ or a vertex $v$ certifying that  $\pi$ is not $\M$-independent or a vertex set $X$ certifying that $(D,\S,\pi)$ is not $\M$-connected can be found in polynomial time.
\end{theo}
\begin{proof}
 By the submodularity of $\rho_D(X) + r_{\M}(\S_X)$, \ref{minimizing_sub}, using the oracle on $\M$ and \ref{packing}, we can either find a set violating \eqref{cond2} or a vertex certifying that $\pi$ is not $\M$-independent or certify that there exists an $\M$-basic packing of arborescences.

 In the latter case, an $\M$-basic packing of arborescences can be found in polynomial time following the proof of \ref{packing}. Using the oracle, test whether each arc is good or bad. When an arc $uv$ is bad, for each $s \in \S_u \setminus Span(\S_v)$, determine in polynomial time whether $D'$ is $\M'$-connected using the submodularity of $\rho_{D'}(X) + r_{\M'}(\S'_X)$, the oracle for the rank function  $r_{\M'}$ (that is easily computed from  $r_{\M}$) and \ref{minimizing_sub}. Either all arcs are good or we find a bad arc  $uv$ and $s \in \S_u \setminus Span(\S_v)$ satisfying \ref{nice_bad_arc}. In the first case, by \ref{arcs_are_good}, the required packing is found. In the second case, it leads to the computation of an $\M'$-basic packing in the digraph with roots $(D',S',\pi')$ which contains less arcs than $D$.
\end{proof}

By the submodularity of $x(\varrho_D(X)) + r_{\M}(\S_X)$ and \ref{minimizing_sub}, $P_{\M,D}$ can be separated in polynomial time.
Thus, using the ellipsoid method, by Gr\"otschel, Lov\'asz and Schrijver \cite{separation}, and by \ref{packing_polytime}, we have the following result.
\begin{theo}
 Let $(D,{\S},\pi)$ be a digraph with roots,  $\M$  a matroid on ${\S}$ and $c$ a cost function on the set of arcs of $D.$ If there exists an $\M$-basic packing of arborescences in $(D,{\S},\pi)$ then one of minimum cost can be found in polynomial time.
\end{theo}

We conclude this section with algoritmic remarks on the undirected case. Let $(G,{\S},\pi)$ be a graph with roots and $\M$  a matroid on $\S$. Katoh and Tanigawa \cite{Kato_Tanigawa} designed a combinatorial algorithm to decide in polynomial time whether $(G,\S,\pi)$ admits an $\M$-basic packing of rooted trees such that the edge set of $G$ is the union of the edge sets of the trees in the packing and, if it does, find the decomposition. As far as we know, their algorithm does not find an $\M$-basic packing of rooted-trees in $(G,\S,\pi)$ in the general case (where the condition on the edges is deleted). However, our approach gives a polynomial time algorithm to solve this problem. Indeed, if $(G,\S,\pi)$ is $\M$-partition connected, then an orientation $D$ of $G$ such that  $(D,\S,\pi)$ is $\M$-connected can be found in polynomial time using submodular flows \cite{frank_submodular}. By \ref{packing_polytime}, an $\M$-basic packing of arborescences of $(D,\S,\pi)$, and hence an $\M$-basic packing of rooted trees of $(G,\S,\pi)$, can be found in polynomial time.


\section{Final remarks} 

We finish the paper with a related problem.  
Given   a digraph with roots $(D,\S,\pi)$,  a matroid $\M$ on ${\S}$ with rank function $r_{\M}$ and a bound $b:V\rightarrow \mathbb Z$, an \emph{$(\M,b)$-packing of arborescences} is a set  $\{T_1,\dots ,T_t\}$ of pairwise arc-disjoint arborescences  such that   $T_i$ is rooted at ${\s}_i \in {\S}$  for  $i=1,\dots ,t$ and  $r_{\M}(\{{\s}_i \in {\S} : v \in V(T_i)\}) \geq b(v)$ for all $v \in V$. When $b$ is constant, using \ref{packing} and matroid truncation, one can derive a characterization of digraphs with roots admitting an $(\M,b)$-packing of arborescences.
On the other hand, for general $b$, the problem turns out to be NP-complete since it contains the disjoint Steiner arborescences problem that is to find $2$ arc-disjoint arborescences both rooted at the same  vertex and both covering a specified subset of vertices.


\section{Acknowledgement}

We are grateful to Andr\'as Frank for some remarks on an earlier version of the paper and 
to Yohann Benchetrit for his invaluable discussion on the polyhedral aspects.


\begin{thebibliography}{nn}
\bibitem {edmonds} J. Edmonds, Edge-disjoint branchings, in: ed. B. Rustin, Combinatorial Algorithms,
Academic Press, New York,  (1973) 91-6
\bibitem {frank1} A. Frank, On disjoint trees and arborescences, in: Algebraic Methods in Graph Theory,
Colloquia Mathematica Soc. J. Bolyai, 25 (1978) 159-69
\bibitem {frank2} A. Frank, Kernel systems of directed graphs, Acta Scientiarium Mathematicarum (Szeged), 41 (1-2) (1979) 63-76
\bibitem {frank} A. Frank,  On the orientation of graphs. J. Comb. Theory, Ser. B 28 (3) (1980) 251-261
\bibitem {frank_submodular} A. Frank, An algorithm for submodular functions on graphs, Annals of Discrete Mathematics 16 (1982) 97-120
\bibitem {frankbook} A. Frank, Connections in combinatorial optimization. Oxford Lecture Series in Mathematics and its Applications, 38. Oxford University Press, Oxford, 2011
\bibitem {separation} M. Gr\"otschel, L. Lov\'asz, A. Schrijver, The ellipsoid method and its consequences in combinatorial optimization, Combinatorica, Springer Berlin, 1 (2) (1981) 169-197
\bibitem {Iwata} S. Iwata, L. Fleischer, S. Fujishige, A combinatorial strongly polynomial algorithm for minimizing submodular functions. J. ACM 48, 4  (2001) 761-777
\bibitem {Kato_Tanigawa} N. Katoh, S. Tanigawa, Rooted-tree Decompositions with Matroid Constraints and the Infinitesimal Rigidity of Frameworks with Boundaries, manuscript, 2011
\bibitem {NW} C.St.J.A. Nash-Williams, Edge-disjoint spanning trees of finite graphs, J. London Math.
Soc., 36 (1961) 445-450 
\bibitem {Tutte} W.T. Tutte, On the problem of decomposing a graph into $n$ connected factors, J. London
Math. Soc., 36 (1961) 221-230
\bibitem {schrijver} A. Schrijver, A combinatorial algorithm minimizing submodular functions in strongly polynomial time, J. Combin. Theory, Ser. B 80 (2000) 346-355

\end{thebibliography}

\end{document}